\documentclass[12pt, a4paper]{amsart}

\usepackage{amsmath, amssymb, amsthm, amsfonts}
\usepackage{mathrsfs}
\usepackage{mathtools}
\usepackage[top=2.6cm, bottom=2.5cm, left=2.6cm, right=2.6cm]{geometry}

\theoremstyle{plain}
\newtheorem{thm}{Theorem}

\newtheorem{lem}{Lemma}

\newtheorem{prop}{Proposition}
 
\theoremstyle{remark}
\newtheorem{remark}{Remark}

\begin{document}
\title{The special unitary groups $SU(2n)$ as framed manifolds} 
\author{Haruo Minami}
\address{H. Minami: Professor Emeritus, Nara University of Education}
\email{hminami@camel.plala.or.jp}
\begin{abstract}
Let $[SU(2n), \mathscr{L}]$ denote the bordism class of $SU(2n)$ $(n\ge 2)$ equipped with its left invariant framing $\mathscr{L}$. Then it is well known that $e_\mathbb{C}([SU(2n), \mathscr{L}])=0$ where $e_\mathbb{C}$ denotes the complex Adams $e$-invariant. In this note we show that replacing $\mathscr{L}$ by the framing obtained by twisting it by a specific 
map the zero value of $e_\mathbb{C}([SU(2n), \mathscr{L}])$ can be transformed into a generator of $\mathrm{Im} \, e_\mathbb{C}$ which is isomorphic to a cyclic group. In addition we show that the same procedure affords an analogous result for a quotient of $SU(2n+1)$ by a circle subgroup which inherits a canonical framing from $SU(2n+1)$ in the usual way. 
\end{abstract}

\maketitle

\section{Introduction}

Let $G$ be a simply-connected compact Lie group of dimension $4\mathit{l}-1$ 
and of rank $\ge 2$. Let us consider it to be the framed manifold equipped with the left invariant framing $\mathscr{L}$ and write $[G, \mathscr{L}]$ for its bordism class in 
$\pi_{4\mathit{l}-1}^S$. Then it is well known ~\cite{AS} that 
\[e_\mathbb{C}([G, \mathscr{L}])=0\] 
where $e_\mathbb{C} : \pi_{4\mathit{l}-1}^S \to \mathbb{Q}/\mathbb{Z}$ 
denotes the complex Adams $e$-invariant. In view of this result we wish to raise the problem of the existence of a map $f : G\to GL(s, \mathbb{R})$ 
such that $e_\mathbb{C}([G, \mathscr{L}^f])$ is nonzero and if possible, is equal to 
$(-1)^{l-1}B_l/2l$  ~\cite{A} where $\mathscr{L}^f$ denotes the framing obtained by twisting $\mathscr{L}$ by $f$ (\!\!~\cite{AGP},\!~\cite{K},\!~\cite{S}) and $B_l$ the $l$th Bernoulli number.

In this note we consider the case $G=SU(2n)$ with dimension $4n^2-1$
$(n\ge 2)$. The main result is 
the following theorem.
\begin{thm} Let $\rho : SU(2n)\to GL(4n, \mathbb{R})$ $(n\ge 2)$ be 
the standard real representation of $SU(2n)$. Then we have
\[
e_\mathbb{C}\bigl([SU(2n), \mathscr{L}^{(n-1)\rho}]\bigr)=(-1)^{n-1}B_{n^2}/2n^2
\]
\end{thm}
The proof is based on Proposition 2.1 of ~\cite{LS}. But also in the case 
$G=SU(2n+1)$ it can be seen that the procedure used there is adaptable to 
its quotient by a circle subgroup $C$, considered as a framed manifold equipped with the induced framing $\mathscr{L}_C$. Now this framing can be obtained by applying the fact that $\mathscr{L}$ is a $C$-equivariant framing ~\cite{K} to 
the equation $(\ast)$ of  ~\cite[p. \!\!16]{LS}. Let 
$S^1=\{z\in\mathbb{C}| |z|=1\}$ and put 
\[
C=\{\mathrm{diag}(1, \ldots, 1, z, \bar{z})\mid z\in S^1\}
\]
where 
$\mathrm{diag}(a_1, \ldots, a_{2n+1})$ denotes the diagonal matrix with 
diagonal entries $a_1, \ldots, a_{2n+1}$. Then we have
\begin{prop} 
$e_\mathbb{C}\bigl([SU(2n+1)/C, \mathscr{L}_C]\bigr)
=-B_{n^2+n}/2(n^2+n)$\qquad $(n\ge 1)$.
\end{prop}  

\section{Tensor product decomposition of $E$}

Let $G=SU(m)$ $(m\ge 2)$ and $S$ be the circle subgroup of $G$  
generated by $d(z)=\mathrm{diag}(\bar{z}^{m-1}, z, \ldots, z)$ with 
$z\in S^1$. Regard $G$ as an $S$-space endowed with $d(z)$ 
acting on $g\in G$ by the rule $(g, z)\to gd(z)$ and consider it to be the principal bundle along with the natural projection $p : G\to G/S$.  Let $E=G\times_S\mathbb{C}$ denote the canonical complex line bundle over $G/S$ associated to $p$ where $S$ acts on $\mathbb{C}$ as $S^1$. Then its unit 
sphere bundle $\pi : S(E)\to G/S$ is naturally isomorphic to $p : G\to G/S$ as 
a principal $S$-bundle over $G/S$, which is expressed as 
$(S(E), \pi, G/S)\cong (G, p, G/S)$ as usual.

In order to introduce some further notations we first recall the case $m=2$. Write  
\begin{equation*}
R(rz, v)=\begin{pmatrix} rz & v \\
-\bar{v} & r\bar{z} \end{pmatrix}, \quad \ r\ge 0, \ z\in S^1, \ v\in\mathbb{C},  
\end{equation*}
for the elements of $SU(2)$. Then $R(rz, v)d(z)=R(r, zv)$ and so letting it correspond to the element $(1-2r^2, 2rzv)\in S^2$ we have a homeomorphism 
between $SU(2)/S$ and $S^2$. Thinking of $p$ as a principal $S$-bundle over $S^2$ via this homeomorphism we have $(S(L), \pi, S^2)\cong (SU(2), p, S^2)$ where $L$ is used instead of $E$. From the formula of  ~\cite[p.\,40]{LS} we have 
$e([S(L), \Phi_L])=B_1/4$ where $\Phi_L$ denotes the trivialization of the stable tangent bundle of $S(L)$. In this case $\Phi_L$ coincides with $\mathscr{L}$ via the above isomorphism and so we have $e([SU(2),\mathscr{L}])=1/24$. 

For brevity we write $(r, zv)_R$ for $(1-2r^2, 2rzv)\in S^2$ and identify 
\[(r, zv)_R=p(R(rz, v))=R(r, zv) \] 
where $zv$ can be converted to 1 when $r=0$ because    
$R(0, zv)\,d(\overline{zv})=R(0, 1)$. 
Thus $(r, zv)_R$ represents $(0, 1)_R$ if $r=0$ and obviously $(1, 0)_R$ if $r=1$. The latter is assumed to be specified as the base point of $S^2$ for the reason that it is the image of the identity matrix of $SU(2)$ by $p$.

Suppose $m\ge 3$ and let $0\le j \le m-2$ and $1\le i \le m_j$ where 
$m_j=m-j-1$. Further we assume that the factors used in the product constructions are arranged in ascending order of their suffixes derived by $i, j$. For fixed $j$ we put       
\begin{equation*}
R_{i; j}(r_{i; j}z, v_{i; j})=
\begin{pmatrix} 
I_j & 0 & 0 & 0 & 0 \\ 
0 & r_{i; j}z & 0 & v_{i; j} & 0 \\ 
0 & 0 & I_{i-1} & 0 & 0 \\
0 & -\bar{v}_{i; j}& 0 & r_{i; j}\bar{z} & 0\\
0 & 0 & 0 & 0 & I_{m_j-i}
\end{pmatrix} \ \text{with} \ R(r_{i; j}z, v_{i; j})\in SU(2)
\end{equation*} 
where $I_s$ denotes the identity matrix of order $s$. Then we have  
\begin{equation}
\textstyle\prod_{i=1}^{m_j}R_{i; j} (r_{i; j}z, v_{i; j})=
\left(
\begin{array}{cccccc} 
I_j & 0 & 0 & 0 & \cdots & 0 \\ 
0 & a_{0; j} & b_{1; j} & b_{2; j} &\cdots & b_{m_j; j} \\ 
0 & a_{1; j} & c_{1, 1; j}  & c_{1, 2; j} & \cdots & c_{1, m_j; j}  \\
0 & a_{2; j} & 0& c_{2, 2; j} & \cdots & c_{2, m_j; j}  \\
\vdots & \vdots  & \vdots  &\ddots  & \ddots  &   \vdots\\
0 & a_{m_j; j} & 0 & \cdots & 0 & c_{m_j, m_j; j} 
\end{array}
\right)
\end{equation}
where
\begin{equation*}
\begin{split}
&a_{0; j}=r_{1; j}\cdots r_{m_j; j}z^{m_j},\\
&a_{s; j}=-r_{s+1; j}\cdots r_{m_j; j}z^{m_j-s}\bar{v}_{s; j} 
\quad (1\le s\le m_j-1), \hspace{35mm}\\ 
&a_{m_j; j}=-\bar{v}_{m_j; j},\\
\end{split}
\end{equation*}
\begin{equation*}
\begin{split}
&b_{1; j}=v_{1; j},\\
&b_{s; j}=r_{1; j}\cdots r_{s-1; j}z^{s-1}v_{s; j} \quad (2\le s\le m_j),\\ 
&c_{s, t; j}=-r_{s+1; j}\cdots r_{t-1; j}z^{t-s-1}
\bar{v}_{s; j}v_{t; j}\quad (1\le s\le t-2, \ 3\le t \le m_j),\hspace{7mm}\\ 
&c_{s, s+1; j}=-\bar{v}_{s; j}v_{s+1; j} \quad (1\le s\le m_j-1), \\
&c_{s, s; j}=r_{s; j}\bar{z} \quad (1\le s\le m_j), \\  
&c_{s, t; j}=0 \quad (s>t).\\
\end{split}
\end{equation*}

Let $d_j(z)=\mathrm{diag}(1, \overset{(j)}\ldots, 1, \bar{z}^{m_j}, z, \ldots, z)$. 
By referring to the above calculation we have 
\begin{equation}
\biggl(\textstyle\prod_{i=1}^{m_j}R_{i; j} (r_{i; j}z, v_{i; j})\biggr)
d_j(z)=\prod_{i=1}^{m_j}R_{i; j} (r_{i; j}, z^iv_{i; j}).
\end{equation}
The matrix representation of the right-hand side of this equation has a similar 
form to that of (1). In fact, restricting ourselves to the part we need for later 
use we have
\begin{equation}
\begin{split}
&\underline{a}_{0; j}=r_{1; j}\cdots r_{m_j; j}, \ \\
&\underline{a}_{s; j}=-r_{s+1; j}\cdots r_{m_j; j}\bar{w}_{s; j} \ \ 
(1\le s\le m_j-1), \hspace{45mm}\\ 
&\underline{a}_{m_j; j}=-\bar{w}_{m_j; j}, \\ 
\end{split}
\end{equation}
where $\underline{a}_{s; j}$ denotes its matrix element corresponding to 
$a_{s; j}$ and $w_{s; j}$ stands for $z^sv_{s; j}$.

Let $D_0(z)=I_m$ and $D_j(z)=\mathrm{diag}(z^{m-1}, \bar{z}, 
\overset{(j-1)}\ldots, \bar{z}, \bar{z}^{m-j}, 1, \ldots, 1)$ 
$(1\le j\le m-2)$. Clearly $D_j(z)d(z)(=d(z)D_j(z))=d_j(z)$. Define 
$R^{\{i\}}_j(r_{i; j}z, v_{i; j})$ to be the product
\[
\biggl(\textstyle\prod_{s=1}^{m_j}R_{s; j} (r_{s; j}z, v_{s; j})\biggr)D_j(z) \ \text{with} \ r_{s; j}=1 \ \text{for all} \ s, j \ \text{except} \ s=i.
\] 
Then by (2) we have
\begin{equation}
R^{\{i\}}_j(r_{i; j}z, v_{i; j})d(z)=R_{i; j}(r_{i; j}, z^{i}v_{i; j}), \qquad z\in S^1
\end{equation}
since $R_{s; j}(r_{s; j}, z^sv_{s j})=I_m$ if $r_{s; j}=1$. Hence putting 
\[
R_j(r_{i; j}z_i, v_{i; j})=\textstyle\prod_{i=1}^{m_j}
R^{\{i\}}_j(r_{i; j}z_i, v_{i; j}), \qquad z_i\in S^1.
\] 
we have due to the commutativeness of $d(z')$ and 
$R^{\{i\}}_j(r_{i; j}z, v_{i; j})$ $(z, z'\in S^1)$ 
\begin{equation}
R_j(r_{i; j}z_i, v_{i; j})d(z_1)\cdots d(z_{m_j})
=\textstyle\prod_{i=1}^{m_j}R_{i; j}(r_{i; j}, z_i^iv_{i; j})\qquad(j\ge 1).
\end{equation}
But in the case of $j=0$, differing from this, we have 
\begin{equation*}\tag{5b}
R_0(r_{i; 0}z_i, v_{i; 0})d(z_1)\cdots d(z_{m-1})
=R_{1; 0}(r_{1; 0}, z_1v_{1; 0})\textstyle\prod_{i=2}^{m-1}
R_{i; 0}(r_{i; 0}, z_{i-1}^mz_i^iv_{i; 0}).
\end{equation*}  
This can be shown using induction. In fact, as in (5), we first have 
\[R_0(r_{i; 0}z_i, v_{i; 0})d(z_{m-1})
=\biggl(\textstyle\prod_{s=1}^{m-2}R_{i; 0}(r_{i; 0}z_i, v_{i; 0})\biggr)
R_{m-1; 0}(r_{m-1; 0}, z_{m-1}^{m-1}v_{m-1; 0}).\]
By letting $d(z_{m-2})$ act on both sides of this it can be derived that  
\begin{equation*}
\begin{split}
R_0(r_{i; 0}z_i, &v_{i; 0})d(z_{m-1})d(z_{m-2})
= \biggl(\textstyle\prod_{i=1}^{m-3}R_{i; 0}(r_{i; 0}z_i, v_{i; 0})\biggr)\\
&R_{m-2; 0}(r_{m-2; 0}, z_{m-2}^{m-2}v_{m-2; 0})
R_{m-1; 0}(r_{m-1; 0}, z_{m-2}^mz_{m-1}^{m-1}v_{m-1; 0}).
\end{split}
\end{equation*}
Repeating this procedure we can achieve the equation (5b). Further 
these equations have the following extensions. 
\begin{equation}
\begin{split}
R^{\{i\}}_j(r_{i; j}zz', v_{i; j})d(z)&=R^{\{i\}}_j(r_{i; j}z'_, z^iv_{i; j})\quad 
(j\ge 0), \\
R_j(r_{i; j}z_iz'_i, v_{i; j})d(z_1)\cdots d(z_{m_j})&=
\begin{cases}
R_j(r_{i; j}z'_i, z_i^iv_{i; j}) &  (j\ge 1), \\
R_j(r_{i; j}z'_i, z_{i-1}^mz_i^iv_{i; j})& (j=0)\quad (\text{where} \ z_0=1).
\end{cases}
\end{split}
\end{equation}

From (4) we see that $p(R^{\{i\}}_j (r_{i; j}z, v_{i; j}))=p(R_{i; j}(r_{i; j}, z^iv_{i; j}))$ holds for $p : G\to G/S$.  Letting $S^2_{i; j}$ denote the 2-sphere consisting of all
$(r_{i; j}, w_{i; j})_R$, this permits us to embed $S^2_{i; j}$ in $G/S$ by means of the injective map $(r_{i; j}, w_{i; j})_R\to p(R^{\{i\}}_j (r_{i; j}z, v_{i; j}))$. From (6) we also see that
\[P^{\{i\}}_j=\{R^{\{i\}}_j (r_{i; j}z, v_{i; j})| (r_{i; j}, w_{i; j})_R\in S^2, z\in S^1|\}
\subset G\] is closed under the action of $S$.
Combining these two facts we have that 
$q_{i; j}=p|P^{\{i\}}_j : P^{\{i\}}_j\to S^2_{i; j}\subset G/S$ provides the projection 
map of a principal $S$-bundle over $S^2_{i; j}$. Then $q_{i; j}$ factors through the map $P^{\{i\}}_j\to SU(2)$ given by 
$R^{\{i\}}_j (r_{i; j}z, v_{i; j})\to R(r_{i; j}z, z^{i-1}v_{i; j})$ which yields an isomorphism $(P^{\{i\}}_j, q_{i; j}, S^2_{i; j})\cong (SU(2),  p, S^2)$. Let us write $L_{i; j}$ for the canonical complex line bundle over $S^2_{i; j}$ associated to $q_{i; j}$. Then since  
$(S(L_{i; j}), \pi, S^2_{i; j})\cong (P^{\{i\}}_j, q_{i; j}, S^2_{i; j})$ we have 
\begin{equation}
(S(L_{i; j}), \pi, S^2_{i; j})\cong (SU(2), p, S^2)\quad \ 
(\text{identifying} \ S^2_{i; j}=S^2).
\end{equation}

Let $(S_j^2)^{m_j}=S^2_{1; j}\times\cdots\times S^2_{m_j; j}$ and 
$(S^2)^{(m^2-m)/2 }=(S_0^2)^{m_0}\times\cdots\times (S_{m-2}^2)^{m_{m-2}}$. 
Then we have a map $\phi : (S^2)^{(m^2-m)/2 }\to G/S$ given by
\begin{equation*} 
x=(x_0, \ldots, x_{m-2})\to p\biggl(\textstyle\prod_{j=0}^{m-2}R_j(r_{i; j}, w_{i; j})\biggr)
\end{equation*}
where $x_j=(x_{1; j}, \ldots, x_{m_j; j})\in (S_j^2)^{m_j}$ with 
$x_{i; j}=(r_{i; j}, w_{i; j})_R\in S^2_{i; j}$. 
If we put $R(r_{i; j}z_i, v_{i; j})=\prod_{j=0}^{m-2}R_j(r_{i; j}z_i, v_{i; j})$, then  as 
in the case of $P^{\{i\}}_j$, using (5), (5b), (6) we see that
\[P=\{R(r_{i; j}z_i, v_{i; j}) | (r_{i; j}, w_{i; j})_R\in S^2_{i; j}, z_i\in S^1\}
\subset G\] forms the total space of a principal $S$-bundle endowed with the projection map $q : P\to (S^2)^{(m^2-m)/2}$ such that $\phi\circ q=p|P$. 
If we write $L^{(m^2-m)/2}=L_0^{m_0}\boxtimes\cdots\boxtimes 
L_{m-2}^{m_{m-2}}$ where $L_j^{m_j}=L_{1; j}\boxtimes\cdots\boxtimes L_{m_j; j}$,
then $(S(L^{(m^2-m)/2}), \pi, (S^2)^{(m^2-m)/2})\cong 
(P, q, (S^2)^{(m^2-m)/2})$ for the reason similar to (7). Therefore we have
\begin{equation}
(S(L^{(m^2-m)/2}), \pi, (S^2)^{(m^2-m)/2})\cong\phi^*(S(E), \pi, G/S), 
\end{equation}
where $\phi^*(-)$ denotes the induced bundle by $\phi$. 

Let $(S^2_{i; j})^\circ\in S^2_{i; j}$ be the subspace 
consisting of $(r_{i; j}, w_{i; j})_R$ with $r_{i; j}>0$, and put
$((S^2)^{m_j})^\circ=\prod_{i=1}^{m_j}(S^2_{i; j})^\circ$ and 
$((S^2)^{(m^2-m)/2})^\circ=\prod_{j=0}^{m-2}((S^2)^{m_j})^\circ$.  Then 

\begin{lem} The restriction of $\phi$ to $((S^2)^{(m^2-m)/2})^\circ$
is injective. 
\end{lem}
\begin{proof}  By definition 
$\phi(x)=p(\prod_{j=0}^{m-2}R_j(r_{i; j}, w_{i; j}))$ and  
$R_j(r_{i; j}, w_{i; j})=\textstyle\prod_{i=1}^{m_j}R_{i; j}(r_{i; j}, w_{i; j})$. 
Here $R_j(r_{i; j}, w_{i; j})$ can be written in the form 
$R_j(r_{i; j}, w_{i; j})= \mathrm{diag}(I_j, M_j)$ where the first column of $M_j$ 
consists of $\underline{a}_{0; j}$, $\ldots$, $\underline{a}_{m-j-1; j}$ given in (3).

Suppose $\phi(x)=\phi(x')$, namely $p(\prod_{j=0}^{m-2}R_j(r_{i; j}, w_{i; j}))
=p(\prod_{j=0}^{m-2}R_j(r'_{i; j}, w'_{i; j}))$ where we denote by attaching $ ' $ to an element accompanied by $x$ its corresponding element accompanied by $x'$. 
Due to the definition this can be interpreted as meaning that 
\begin{equation*}\tag{$\ast$}
R_0(r_{i; 0}, w_{i; 0})\cdots R_{m-2}(r_{i; m-2}, w_{i; m-2})
=R_0(r'_{i; 0}, w'_{i; 0})\cdots R_{m-2}(r'_{i; m-2}, w'_{i; m-2})
\end{equation*} 
with $w_{i; j}$ and $w'_{i; j}$ converted to 1 when $r_{i; j}=0$ and $r'_{i; j}=0$. 
This attached condition can be verified as in the case of $SU(2)$ as follows. Consider the latter equation of (6) with $z'_i=1$ and $v_{i; j}$ replaced by 
$w_{i; j}$ for all $i$. Let $r_{i; j}=0$ for $i=i_1, \ldots, i_s$. Then $|w_{i; j}|=1$ 
for $i=i_1, \ldots, i_s$, so taking $z_i=\bar{w}_{i; j}^{1/i}$ for $i=i_1, \ldots, i_s$ 
and $z_i=1$ except for these $z_i$, we have 
$R_j(r_{i; j}z_i, w_{i; j})d(z_1)\cdots d(z_{m_j})=R_j(r_{i; j}, z_i^iw_{i; j})$ 
where $R_{i; j}(r_{i; j}, w_{i; j})=R_{i; j}(0. 1)$ for $i=i_1, \ldots, i_s$. But in the case 
$j=0$, if $r_{i-1; 0}\ne 0$, then the above $z_i$ needs to be replaced 
by $z_i=w_{i-1; 0}^m\bar{w}_{i; j}^{1/i}$. Similarly for $w'_{s; j}$.

In view of the form of $R_j(r_{i; j}, w_{i; j})$ we have 
\[\underline{a}_{0; 0}=\underline{a}'_{0; 0}, \ldots, \underline{a}_{m-1; 0}=\underline{a}'_{m-1; 0}.\]
Substituting the result of (3) into these equations we have 
$r_{1; 0}\cdots r_{m-1; 0}=r'_{1; 0}\cdots r'_{m-1; 0}$, $r_{s+1; 0}\cdots r_{m-1; 0}\bar{w}_{s; 0}=r'_{s+1; 0}\cdots r'_{m-1; 0}\bar{w}'_{s; 0}$ $(1\le s\le m-2)$ and $\bar{w}_{m-1; 0}=\bar{w}'_{m-1; 0}$ where $r_{s; 0}^2+|w_{s; 0}|^2=1$ and 
$r_{s; 0}> 0$ by the assumption. We now proceed by induction on $i$ in reverse order. First from $\underline{a}_{m-1; 0}=\underline{a}'_{m-1; 0}$ 
we have $w_{m-1; 0}=w'_{m-1; 0}$ and so it follows that $r_{m-1; 0}=r'_{m-1; 0}$. Next from the second equation  
$\underline{a}_{m-2; 0}=\underline{a}'_{m-2; 0}$ we see that 
$r_{m-1; 0}w_{m-2; 0}=r'_{m-1; 0}w'_{m-2; 0}$. But since  
$r_{m-1; 0}=r'_{m-1; 0}> 0$ we have $w_{m-2; 0}=w'_{m-2; 0}$ and so similarly it follows that  $r_{m-2; 0}=r'_{m-2; 0}$. Repeating this procedure we obtain subsequently $w_{m-3; 0}=w'_{m-3; 0}$, $r_{m-3; 0}=r'_{m-3; 0}$, 
$w_{m-4; 0}=w'_{m-4; 0}$, $r_{m-4; 0}=r'_{m-4; 0}$,$\cdots$, $r_{2; 0}=r'_{2; 0}$ and 
finally we obtain $w_{1; 0}=w'_{1; 0}$, $r_{1; 0}=r'_{1; 0}$. At this point we conclude that $x_0=x'_0$ and so the equation $(\ast)$ above can be rewritten as
\[R_1(r_{i; 0}, w_{i; 0})\cdots R_{m-2}(r_{i; m-2}, w_{i; m-2})
=R_1(r'_{i; 0}, w'_{i; 0})\cdots R_{m-2}(r'_{i; m-2}, w'_{i; m-2})\]
with $w_{i; j}=1$ when $r_{i; j}=0$ and $w'_{i; j}=1$ when $r'_{i; j}=0$.
From this by looking at the form of $R_j(r_{i; j}, w_{i; j})$ again we have
$\underline{a}_{0; 1}=\underline{a}'_{0; 1}, \ldots, \underline{a}_{m-2; 1}=\underline{a}'_{m-2; 1}$.  Replacing the previous equations by these ones 
and the condition $r_{i; 0}> 0$ by $r_{i; 1}> 0$ in the argument there we can get $x_1=x'_1$, hence the rewritten assumption above is also rewritten as 
\[R_2(r_{i; 0}, w_{i; 0})\cdots R_{m-2}(r_{i; m-2}, w_{i; m-2})
=R_2(r'_{i; 0}, w'_{i; 0})\cdots R_{m-2}(r'_{i; m-2}, w'_{i; m-2})\]
with the same condition as above.
Repeating this step we obtain $x_2=x'_2$, $\ldots$, $x_{m-2}=x'_{m-2}$ via inductive reduction, which implies $x=x'$. This proves 
the lemma. 
\end{proof}

\section{Proofs of Theorem and Proposition}

For any $k\in\mathbb{Z}$ such that $1\le k\le m/2-1$ we put 
\[D(z_{2k-1}, zz_{2k})=\mathrm{diag}(1, \overset{(2k-1)}\ldots, 1, z_{2k-1}, 
\bar{z}_{2k-1}zz_{2k}, \bar{z}\bar{z}_{2k}, 1, \ldots, 1)\] 
for $z_{2k-1}, z_{2k}, z\in S^1$. 
Let $Q_k$ be the subspace of $G$ consisting of $D(z_{2k-1}, zz_{2k})d(\bar{z})$. Then it forms the total space of a principal $S$-bundle over $T_k=S^1\times S^1$ along with the projective map $p_k : Q_k\to T_k$ given by $D(z_{2k-1}, zz_{2k})d(\bar{z})\to (z_{2k-1}, zz_{2k})$ where $T_k$ is considered as a 
subspace of $G/S$ under $\psi_k : (z_{2k-1}, zz_{2k})\to p(D(z_{2k-1}, zz_{2k}))$. 

Put $z_{2k-1}=e^{\eta i}$ and 
$z_{2k}=e^{\theta i}$ for $0\le \eta, \, \theta < 2\pi$ and define a map 
$\mu : T_k\to S^2$ by 
\begin{equation*}
\mu(e^{\eta i}, ze^{\theta i}) = \left\{
\begin{array}{ll}
(\cos(\eta/2),  ze^{\theta i}\sin(\eta/2))_R & \ (0\le \eta\le \pi)\\
(-\cos(\eta/2),  ze^{\theta ti}\sin(\eta/2))_R & \ (\pi\le \eta< 2\pi), \ \ \ 
t=2-\eta/\pi. 
\end{array}
\right.
\end{equation*}
Then taking into account the fact that a principal circle bundle over $S^1$ is trivial we see that the classifying map of $p_k$ factors through $S^2$ where the restriction of $p_k$ to $\{1\}\times S^1\subset T_k$ is viewed as being trivial. 
Therefore we have 
\begin{lem}[cf.\,\cite{LS}, \!\S2, Example 3] 
$(Q_k, p_k, T_k)\cong \mu^*(SU(2), p, S^2)$ and also $\mu$ induces an 
isomorphism 
$H^2(S^2, \mathbb{Z})\cong H^2(T_k, \mathbb{Z})$ 
$(1\le k\le m/2-1)$.  
\end{lem}
\begin{proof}
In order to prove the first equation it suffices to show that there is 
a bundle map covering 
$\mu$. 
Let $\tilde{\mu} : Q_k\to SU(2)$ be the map given by 
\begin{equation*}
\tilde{\mu}(D(e^{\eta i}, ze^{\theta i})d(\bar{z}))= \left\{
\begin{array}{ll}
R(\bar{z}\cos(\eta/2),  e^{\theta i}\sin(\eta/2)) & \ (0\le \eta\le \pi)\\
R(-\bar{z}\cos(\eta/2),  e^{\theta ti}\sin(\eta/2)) & \ (\pi\le\eta<2\pi), \ \ \ 
t=2-\eta/\pi, 
\end{array}
\right.
\end{equation*}
which provides the required bundle map. The second equation is immediate 
from the definition of $\mu$.
\end{proof}

Let $m=2n$ or $2n+1$ and put $T^{n-1}=T_1\times\cdots\times T_{n-1}$. 
Let $\phi_j : x_j \to p(R_j(r_{i; j}, w_{i; j}))$ and $\psi_k : y_k \to p(D(z_{2k-1}, \zeta_{2k}))$ be the injective maps of $(S_j^2)^{m_j}$ and $T_k$ into 
$G/S$, respectively, described above where $x_j=(x_{1; j}, \cdots, x_{m_j; j})$ with  
$x_{i; j}=(r_{i; j}, w_{i; j})_R\in S^2_{i; j}$ and $y_k=(z_{2k-1}, \zeta_{2k})\in T_k$.  
Then by putting for $x=(x_0, \cdots, x_{m-2})$ and $y=(y_1, \cdots, y_{n-1})$ 
\[
\psi(x, y)=p\biggl(\textstyle\prod_{j=0}^{m-2}R_j(r_{i; j}, w_{i; j}))
\textstyle\prod_{k=1}^{n-1}D(z_{2k-1}, \zeta_{2k})\biggl)
\]
we have a map of  $(S^2)^{(m^2-m)/2}\times T^{n-1}$ into $G/S$. 
Let us put 
\[\tilde{R}(r_{i; j}z_i, v_{i; j}; z_{2k-1}, z'_kz_{2k}) =\textstyle\prod_{j=0}^{m-2}R_j(r_{i; j}z_i, v_{i; j})\prod_{k=1}^{n-1}D(z_{2k-1}, z'_kz_{2k})d(\bar{z}'_k).\] 
Here we see by definition that the product terms in the second-half of the right-hand side satisfy $d(\bar{z}')D(z_{2k-1}, z{z'}^2z_{2k})d(\bar{z})
=D(z_{2k-1}, zz_{2k})d(\bar{z})d(\bar{z}')$. Using these equations together with the 
ones of (5), (5b) and (6), as in the case of $P$ above, we see that 
\[\tilde{P}=\{\tilde{R}(r_{i; j}z_i, v_{i; j}; z_{2k-1}, z'_kz_{2k}) | (r_{i; j}, w_{i; j})_R\in 
S^2_{i; j}, z'_k, z_{2k-1}, z_{2k}\in S^1\}
\subset G\] forms the total space of a principal $S$-bundle endowed with the projection map $\tilde{q} : \tilde{P}\to (S^2)^{(m^2-m)/2}\times T^{n-1}$ such that 
$\psi\circ \tilde{q}=p|\tilde{P}$. Let $\tilde{L}_k$ be the complex line bundle over $T_k$ associated to $p_k$ and  
$\tilde{L}^{n-1}=\tilde{L}_1\boxtimes\cdots\boxtimes\tilde{L}_{n-1}$. Then putting 
$\tilde{E}=L^{(m^2-m)/2}\boxtimes\tilde{L}^{n-1}$ we see that its unit sphere bundle $S(\tilde{E})$ is isomorphic to $\tilde{q}$ and therefore by combining with the equation of (8) we have  
\begin{equation}
(S(\tilde{E}), \pi, (S^2)^{(m^2-m)/2}\times T^{n-1})
\cong \psi^*(S(E), \pi, G/S).
\end{equation}

Putting $(T_k)^\circ=T_k-\{1\}\times S^1$ we write
$(T^{n-1})^\circ=(T_1)^\circ\times\cdots\times (T_{n-1})^\circ$. Then
\begin{lem} The restriction of $\psi$ to 
$((S^2)^{(m^2-m)/2})^\circ\times (T^{n-1})^\circ$ is injective. 
\end{lem}
\begin{proof}
In terms of the notation of the proof of Lemma 1 we prove that if we suppose 
that $\psi(x, y)=\psi(x', y')$, then it follows that 
$(r_{i; j}, w_{i; j})_R=(r'_{i; j}, w'_{i; j})_R$ $(0\le j\le m-2)$ and 
$(z_{2k-1}, \zeta_{2k})=(z'_{2k-1}, \zeta'_{2k})$ $(1\le k\le n-1)$, the latter of which is the part added to the proof of injectivity of the restriction map of $\phi$. The proof can be proceeded along the same lines as in the proof of Lemma 1 based on the result there. But in fact there are the following changes in the use of the elements $\bar{w}_{m_j; j}$ each of which is a starting point of the proof for 
$\phi_j$:
$\bar{w}_{m-1; 0}\to \bar{w}_{m-1; 0}$ (i.e. no change), $\bar{w}_{m-2: 1}\to z_1\bar{w}_{m-2: 1}$,  $\bar{w}_{m-3: 2}\to (\bar{z}_1\zeta_2)\bar{w}_{m-3: 2}$ 
and in general $\bar{w}_{m-2k-2; 2k+1} \to 
(z_{2k+1}\bar{\zeta}_{2k})\bar{w}_{m-2k-2; 2k+1}$, 
$\bar{w}_{m-2k-3; 2k+2} \to (\bar{z}_{2k+1}\zeta_{2k+2})\bar{w}_{m-2k-3; 2k+2}$
for $k=1, \ldots, n-2$. Similarly substituting these elements into the equations 
$\underline{a}_{0; j}=\underline{a'}_{0; j}$, $\underline{a}_{1; j}=\underline{a'}_{1; j}$, 
$\ldots$, $\underline{a}_{m-j-1; j}=\underline{a'}_{m-j-1; j}$ we can obtain 
$(x, y)=(x', y')$.  

First from $\bar{w}_{m-1; 0}=\bar{w}'_{m-1; 0}$ 
we have $x_0=x'_0$, which shows that $z_1\bar{w}_{m-2: 1}=z'_1\bar{w}'_{m-2: 1}$,  
so we have $z_1r_{m-2: 1}=z'_1r'_{m-2: 1}$. Since $r_{m-2: 1}>0$, from the 
second equation it follows that $r_{m-2: 1}=r'_{m-2: 1}$ and so $z_1=z'_1$. 
This concludes that $x_1=x'_1$.  $z_1=z'_1$, 
Next from $(\bar{z}_1\zeta_2)\bar{w}_{m-3: 2}=(\bar{z}'_1\zeta'_2)\bar{w}'_{m-3: 2}$ 
we have $\zeta_2\bar{w}_{m-3: 2}=\zeta'_2\bar{w}'_{m-3: 2}$ and so it follows that $x_2=x'_2$, $\zeta_2=\zeta'_2$. Therefore by  
$(z_3\bar{\zeta}_2)\bar{w}_{m-4; 3}=(z'_3\bar{\zeta}'_2)\bar{w}'_{m-4; 3}$ we have
$z_3\bar{w}_{m-4; 3}=z'_3\bar{w}'_{m-4; 3}$ and so in a similar way to the previous  case we have $x_3=x'_3$, $z_3=z'_3$. Subsequently also by $(\bar{z}_3\zeta_4)
\bar{w}_{m-5; 4}= (\bar{z}'_3\zeta'_4)\bar{w}'_{m-5; 4}$ we have 
$\zeta_4\bar{w}_{m-5; 4}= \zeta'_4\bar{w}'_{m-5; 4}$ and  so it follows that 
$x_4=x'_4$, $\zeta_4=\zeta'_4$. 
At this point we can conclude that $x_1=x'_1$, $x_2=x'_2$, $x_3=x'_3$, 
$x_4=x'_4$ and $y_1=y'_1$, $y_2=y'_2$. These tell us that through repetition 
of these procedures we can arrive at the desired result.
\end{proof}
\begin{proof}[Proof of Theorem]
Let $m=2n$ and put $M=(S^2)^{(m^2-m)/2}\times T^{n-1}$. Then obviously 
$\dim M=\dim G/S=4n^2-2$. By construction we know ~\cite{H} that 
$\psi$ can be deformed to an onto map. Taking into account the injectivity result given in Lemma 3 together with this fact we can conclude that 
$\psi$ is a degree one map, that is, $\psi_*([M])=[G/S]$ where 
$[ \ ]$ denotes the fundamental class. So we have
\[
\langle c_1(E)^{2n^2-1}, \, [G/S]\rangle=\langle (c_1(\psi_*E)^{2n^2-1}), \, [M]\rangle. 
\]
Hence by (9) 
\begin{equation*}
\begin{split} 
\langle c_1(E)^{2n^2-1}, \, [G/S]\rangle
&=\langle c_1(L^{2n^2-n}\hat\otimes\tilde{L}^{n-1} ), \, [M] \rangle\\
&=\langle  c_1(L^{2n^2-n}), \, [(S^2)^{(m^2-m)/2}] \rangle\cdot
\langle c_1(\tilde{L}^{n-1}), \, [T^{n-1}]\rangle
\end{split}
\end{equation*}
Substituting this into the equation of Proposition 2.1 of ~\cite{LS} we obtain
\begin{equation*}
e_{\mathbb C}([S(E), \Phi_E])=(-1)^{n-1}B_{n^2}/2n^2
\end{equation*}
where $\Phi_E$ denotes  the trivialization of the stable tangent space of $S(E)$ derived by the framing $\mathscr{L}_S$ on $G/S$ induced by $\mathscr{L}$. 

Now according to the above definition the element $(z_{2k-1}, \zeta_{2k})\in T_k$ 
$(1\le k \le n-1)$ represents $p(D(z_{2k-1}, \zeta_{2k}))$ in $G/S$ via 
$\psi_k$ where $p : G\to G/S$ and also we know that 
$D(z_{2k-1}, \zeta_{2k})$ can be  written as \[D(z_{2k-1}, \zeta_{2k})=R_{1; 2k-1}(z_{2k-1}, 0)R_{1; 2k}(\zeta_{2k}, 0)\] 
and that 
\[
p(R_{1; 2k-1}(z_{2k-1}, 0)R_{1; 2k}(\zeta_{2k}, 0))=\phi(\ast)
\] 
where $\phi$ is as above and $*=((1, 0)_R, \ldots, (1, 0)_R)\in (S^2)^{(m^2-m)/2}$.
From these relations we find that $\Phi_E$ corresponds to 
$\mathscr{L}$ which is trivialized over $T_k\subset G$ for $1\le k\le n-1$, namely that if we let $\rho$ denote the realification of the standard complex representation $G\to GL(m, \mathbb{C})$, then $\Phi_E$ can be taken to be 
$\mathscr{L}^{(n-1)\rho}$. This proves the theorem.
\end{proof}
\begin{proof}[Proof of Proposition] 
Let $m=2n+1$. Let $C$ be as above. Then clearly $S\cap C=\{I_{2n+1}\}$, so $S\times C$ becomes a subgroup of $G$. Letting $r : G\to G/C$ and 
$\tilde{r} : G/C\to G/(S\times C)$ be the quotient maps, the pair $(r, \tilde{r})$ gives a bundle map between principal $S$-bundles $G\to G/S$ and 
$G/C\to G/(S\times C)$. Consider the composition 
$r\circ\iota : \tilde{P}\hookrightarrow G\to G/C$ and 
$\tilde{r}\circ\psi : (S^2)^{(m^2-m)/2}\times T^{n-1} \to G/S\to G/(S\times C)$ 
where $\iota$ denotes the inclusion of $\tilde{P}$ into $G$. Then clearly 
$(\tilde{r}\circ\psi)\circ\tilde{q}=\tilde{r}\circ(r\circ\iota)$ holds and $\dim((S^2)^{(m^2-m)/2}\times T^{n-1})=\dim G/(S\times C)$. 
Following the proof of Lemma 3 we can see that $\tilde{r}\circ\psi$ is a degree one map. Therefore by replacing $\psi$ by 
$\tilde{r}\circ\psi$ in the proof of the theorem above we can get a modified  
form in the case of $m=2n+1$. This proves the proposition.
\end{proof}

\begin{remark} From the proof of the theorem we see that
\[e_\mathbb{C}([SU(2n), \mathscr{L}])=0.\]
\!\!~\cite{AS}. \!\!The doubling of the framing occurred there can be dissolved by thinking of the restriction of $E$ to $T_k$ for every $k$ as a trivial complex line bundle. But instead its first Chern class becomes zero and so, according to Proposition 2.1 of ~\cite{LS}, the value of $e_\mathbb{C}$ must become zero. 
\end{remark}

\end{document}